\documentclass[11pt]{amsart}
\usepackage{amsfonts,latexsym,amsthm,amssymb,graphicx}
\usepackage[all]{xy}
\usepackage{hyperref}

\DeclareFontFamily{OT1}{rsfs}{}
\DeclareFontShape{OT1}{rsfs}{n}{it}{<-> rsfs10}{}
\DeclareMathAlphabet{\mathscr}{OT1}{rsfs}{n}{it}

\newtheorem{theorem}{Theorem}[section]
\newtheorem*{theorem*}{Theorem}
\newtheorem{lemma}[theorem]{Lemma}
\newtheorem{corol}[theorem]{Corollary}
\newtheorem{prop}[theorem]{Proposition}
\newtheorem{claim}[theorem]{Claim}

{\theoremstyle{definition} \newtheorem{defin}[theorem]{Definition}}
{\theoremstyle{remark} \newtheorem{remark}[theorem]{Remark}
\newtheorem{example}[theorem]{Example}}

\newcommand{\Pbb}{{\mathbb{P}}}

\newcommand{\Zbb}{{\mathbb{Z}}}
\newcommand{\uBl}{\underline{B\ell}}

\newcommand{\cE}{{\mathscr E}}

\newcommand{\cL}{{\mathscr L}}
\newcommand{\cO}{{\mathscr O}}
\newcommand{\hi}{{\hat\imath}}
\newcommand{\hZ}{{\widehat Z}}
\newcommand{\uC}{\underline C}
\newcommand{\uE}{\underline E}
\newcommand{\trho}{\tilde\rho}
\newcommand{\Til}[1]{{\widetilde{#1}}}

\newcommand{\cf}{{c_{\text{F}}}}
\newcommand{\qede}{\hfill$\lrcorner$}

\DeclareMathOperator{\rk}{rk}
\DeclareMathOperator{\codim}{codim}
\DeclareMathOperator{\Vol}{Vol}

\title{
The Segre zeta function of an ideal
}
\author{Paolo Aluffi}
\address{
Mathematics Department, 
Florida State University,
Tallahassee FL 32306, U.S.A.
}
\email{aluffi@math.fsu.edu}

\begin{document}

\begin{abstract}
We define a power series associated with a homogeneous ideal in a
polynomial ring, encoding information on the Segre classes defined
by extensions of the ideal in projective spaces of arbitrarily high
dimension. We prove that this power series is rational, with poles
corresponding to generators of the ideal, and with numerator of bounded
degree and with nonnegative coefficients. We also prove that this 
`Segre zeta function' only depends on the integral closure of the ideal.

The results follow from good functoriality properties of the `shadows' of 
rational equivalence classes of projective bundles.  More precise results 
can be given if all homogeneous generators have the same degree, and 
for monomial ideals.

In certain cases, the general description of the Segre zeta function given here 
leads to substantial improvements in the speed of algorithms for the computation 
of Segre classes. We also compute the projective ranks of a nonsingular variety 
in terms of the corresponding zeta function, and we discuss the Segre zeta 
function of a local complete intersection of low codimension in projective space.
\end{abstract}

\maketitle


\section{Introduction}\label{intro}
Let $k$ be a field, and let $I\subseteq k[x_0,\dots,x_n]$ be a
homogeneous ideal.  In this paper we consider a formal power series with
integer coefficients determined by $I$,
\begin{equation}\label{ex:deze}
\zeta_I(t)=\sum_{i\ge 0} a_i t^i\quad,
\end{equation}
which we call the {\em Segre zeta function\/} of $I$.  This series is
characterized by the fact that for all $N\ge n$, the class
\[
(a_0 + a_1 H + \cdots + a_N H^N)\cap [\Pbb^N]\quad,
\]
where $H$ denotes the hyperplane class, equals the push-forward to
$\Pbb^N$ of the Segre class $s(Z_N,\Pbb^N)$ of the subscheme $Z_N$
defined by the extension of $I$ to $k[x_0,\dots, x_N]$. It is not difficult to
verify (Lemma~\ref{lem:wd}) that this prescription does determine a unique 
power series as in~\eqref{ex:deze}. The main result of the paper states that
$\zeta_I(t)$ is rational and gives some {\em a priori\/} information on its
poles and its `numerator'. 

Our main motivation comes from the theory of Segre classes.  The {\em
  Segre class\/} of an embedding of schemes is a basic ingredient in
modern (Fulton-MacPherson) intersection theory, with applications to,
among others, enumerative geometry and the computation of invariants
of singularities.  (See \cite{85k:14004} for a thorough treatment of Segre
classes and for many applications.) In practice, concrete computations of
Segre classes are often very challenging, even in comparatively
simple situations. Interesting applications would follow easily if one
could compute the Segre class $s(Z,\Pbb^n)$ of a scheme in $\Pbb^n$,
but even in this restricted context the range of known techniques to
compute Segre classes is essentially limited to their definition and to
reverse-engineering enumerative consequences. Algorithms for their
computation have been developed along these lines and implemented 
(\cite{MR1956868}, \cite{EJP}, \cite{MR3385954}, \cite{harrisFMP}), but
will only deal with small examples. To our knowledge, the
only class of ideals for which an alternative strategy has been developed
is the case of monomial schemes (\cite{MR3070865}, \cite{Scaiop}, and
see~\S\ref{ss:mono} in this paper); but this case also has limited applicability.

The Segre zeta function~\eqref{ex:deze} may be viewed as a `generating
function' for Segre classes. Therefore, general properties of $\zeta_I(t)$ 
translate into properties of Segre classes, which could lead to more tools 
for their computation and a better understanding of these invariants. 
Our main result is the following (see Theorem~\ref{thm:main} for a slightly
extended version).

\begin{theorem*}
Let $I\subseteq k[x_0,\dots, x_n]$ be a homogeneous ideal, and let
$f_0,\dots, f_r$ be a set of homogeneous generators of $I$. Then the
Segre zeta function $\zeta_I(t)$ is {\em rational,\/} with poles at
$-1/d_j$ for a subset $\{d_j\}_{j\in J}$ of $\{\deg
f_i\/\}_{i=1,\dots,r}$. More precisely,
\[
\zeta_I(t) = \frac{P(t)}{(1+(\deg f_0) t)\cdots (1+(\deg f_r) t)}\quad, 
\]
with $P(t)\in \Zbb[t]$ a polynomial with nonnegative coefficients, 
trailing term of degree $\codim I$, and leading term 
$(\prod_i \deg f_i) t^{r+1}$.
\end{theorem*}

In general not all numbers $-1/\deg f_i$ appear as poles of
$\zeta_I(t)$, even for a minimal set of generators of $I$. For
example, if $f_r$ is integral over $(f_0,\dots, f_{r-1})$, then, with
notation as above, $(1+(\deg f_r)t)$ is a factor
of $P(t)$.  Equivalently, $\zeta_I(t) = \zeta_J(t)$ if $J$ is a
reduction of $I$. This observation leads to a sharper version of 
the main theorem (Proposition~\ref{pro:redux}, Corollary~\ref{cor:ref}).

The main features of the statement may be extracted from basic considerations
in Fulton-MacPherson intersection theory; this is observed in~\S\ref{sec:1min}.
We base a full proof of the main result on certain functoriality statements 
concerning the structure theorem for the Chow group of a projective bundle.
These are given in~\S\ref{sec:funsha} without reference to `Segre zeta functions', 
and seem to us independently interesting. The connection with Segre classes
is given in~\S\ref{sec:shaseg}, and the main theorem is proven 
in~\S\ref{sec:szf}, together with some refinements.
For example, we observe (Corollary~\ref{cor:poles}) that if $-1/d$ is a pole of
$\zeta_I(t)$, then $d$ is an element of the degree sequence of a minimal
homogeneous reduction of $I$.

The theorem has applications to the effective computation of Segre classes.
We illustrate this with examples, given in~\S\ref{sec:szf}. Substantial progress in this 
direction will require
a more explicit description of the `numerator' in $\zeta_I(t)$ corresponding
to a degree sequence (determining the denominator). For example, if the
generators $f_0,\dots, f_r$ form a regular sequence, then this numerator
equals $(\prod_i \deg f_i) t^r$. Our hope is that $\zeta_I(t)$ may be 
described more explicitly using tools from e.g., commutative algebra. 
We provide a rather explicit description in two particular cases, 
in~\S\S\ref{ss:ls} and~\ref{ss:mono}, where the numerators are determined
by enumerative geometry considerations and by the volumes of certain
polytopes. 

It would be helpful to have precise results on the behavior of the zeta 
function under standard operations on ideals. We observe
(see~\S\ref{sec:cha}) that $\zeta_{I'+ I''}(t) = \zeta_{I'}(t) \zeta_{I''}(t)$
if $I'$ and $I''$ satisfy a strong transversality condition. It would also
be interesting to establish precise formulas relating the Segre zeta function
with other invariants of an ideal or of the corresponding schemes. 
In~\S\ref{ss:rks} we show how to compute the projective {\em ranks\/} 
of a nonsingular projective variety from the Segre zeta function of a
defining ideal. The (well-known) fact that the projective dual variety of 
a nonsingular complete intersection is a hypersurface is an immediate
consequence (Example~\ref{exa:CI}). In~\S\ref{sec:HC} we prove that, 
under certain hypotheses, the Segre zeta function of a {\em local\/} 
complete intersection in projective space equals the Segre zeta function of a 
{\em global\/} complete intersection. The truth of Hartshorne's 
conjecture would imply that this is in fact the case for all nonsingular 
subvarieties of low codimension in projective space.

\begin{remark}
The paper~\cite{MR3383478} quotes a paper with the tentative title 
``Rationality of a Segre zeta function'' for results on Segre classes
that are close in spirit to the results proven in this paper. Those results 
may now be found in~\cite{tensored}, which includes a discussion
of the Segre zeta functions of ideals generated by sections of a fixed 
line bundle.
\qede\end{remark}

{\em Acknowledgments.} This work was supported in part by the Simons
foundation and by NSA grant H98230-16-1-0016.  The author is grateful
to Caltech for hospitality while this work was carried out.


\section{Shortcut to rationality}\label{sec:1min}

It is easy to convince oneself that $\zeta_I(t)$ should be rational, as a 
consequence of basic properties of the intersection product defined
in~\cite[Chapter~6]{85k:14004}.
Given a homogeneous ideal $I\subseteq k[x_0,\dots, x_n]$, let $\{f_0,\dots, f_r\}$ 
be any set of homogeneous generators for $I$, and let $d_i=\deg f_i$. Assume 
$n>r$. Let $X_i\subseteq\Pbb^n$ be the hypersurface defined by $f_i=0$ in $\Pbb^n$;
thus $Z=X_0\cap \cdots \cap X_r$ is the subscheme of $\Pbb^n$ defined by $I$.
Consider the fiber diagram
\[
\xymatrix{
Z \ar[r] \ar[d]_\delta & \Pbb^n \ar[d]^\Delta \\
X_0\times \cdots \times X_r \ar[r] & \Pbb^n\times\cdots \times \Pbb^n
}
\]
where $\Delta$ is the diagonal embedding. The intersection product $X_0\cdots X_r$ in
$\Pbb^n$ may be defined as
\[
\{ c(\delta^* N_{X_0\times \cdots \times X_r} \Pbb^n\times\cdots \times \Pbb^n)\cap
s(Z,\Pbb^n)\}_{n-r-1}\quad,
\]
the part of dimension $n-r-1$ in the class within brackets.
By \cite[Example 6.1.6]{85k:14004}, 
\[
\{ c(\delta^* N_{X_0\times \cdots \times X_r} \Pbb^n\times\cdots \times \Pbb^n)\cap
s(Z,\Pbb^n)\}_i = 0
\]
for $i<n-r-1$. Denoting by $H$ the hyperplane class, we have
\[
c(\delta^* N_{X_0\times \cdots \times X_r} 
\Pbb^n\times\cdots \times \Pbb^n)=\prod_i (1+d_i H)\quad;
\]
therefore, letting $\iota:Z\hookrightarrow \Pbb^n$ be the inclusion,
\begin{equation}\label{eq:1min}
\iota_* s(Z,\Pbb^n) = \frac{\big(\prod_i(1+d_i H)\cap \iota_* s(Z,\Pbb^n)\big)}
{\prod_i(1+d_i H)}
\end{equation}
expresses $\iota_* s(Z,\Pbb^n)$ as a `rational function' with poles in the 
set $\{-1/d_i\}_{i=0,\dots, r}$, and whose numerator may only 
have nonzero terms in codimension $\le (r+1)$.

This observation captures the essential features of the rationality result of 
this note, modulo technical details such as the independence on the choices 
made here and the refinements mentioned in the introduction. In the sections 
that follow we will provide a full proof of the main theorem, based on functoriality 
considerations concerning `shadows'.


\section{Functoriality of shadows}\label{sec:funsha}

Let $V$ be a variety, and let $\cE$ be a vector bundle on
$V$. According to the structure theorem for the Chow group
$A_*\Pbb(\cE)$ of the projective bundle $\cE$
(\cite[Theorem~3.3(b)]{85k:14004}), every pure-dimensional class $C\in
A_d \Pbb(\cE)$, may be written uniquely as
\[
C=\sum_{j=0}^r c_1(\cO(1))^j\cap \alpha^* (\uC_{d-r+j})
\]
where $\alpha: \Pbb(\cE)\to V$ is the structure morphism, $r+1$ is the
rank of $\cE$, and $\uC_{i-r+j}\in A_{i-r+j} V$.

\begin{defin}[\cite{MR2097164}] The {\em shadow\/} of $C$ is the class
$\uC:=\uC_{d-r}+\cdots+\uC_d$.
\qede\end{defin}

By definition, $C\in A_d\Pbb(\cE)$ may be reconstructed from its
shadow and from the information of its dimension $d$. On the other
hand, the shadow of $C$ may be computed directly from $C$:
\begin{equation}\label{eq:shad}
\uC=c(\cE)\cap \alpha_*(c(\cO(-1))^{-1}\cap C)
\end{equation}
(\cite[Lemma~4.2]{MR2097164}).

Not surprisingly, shadows satisfy simple compatibility properties with
respect to proper or flat morphisms.  Let $\varphi: W \to V$ be a
morphism. If $\cE$ is a vector bundle on $V$, we have an induced
morphism $\hat \varphi: \Pbb(\varphi^*\cE) \to \Pbb(\cE)$ and the
fiber square
\[
\xymatrix{
\Pbb(\varphi^*\cE) \ar[d]_\beta \ar[r]^{\hat\varphi} & \Pbb(\cE) 
\ar[d]^\alpha \\
W \ar[r]_\varphi & V
}
\]

\begin{lemma}\label{lem:shapf}
Assume $\varphi, \hat\varphi$ are proper, and let $C'$ be a class in
$A_*\Pbb(\varphi^*\cE)$. Then
\[
\varphi_*(\uC')=\underline{\hat\varphi_*(C')}\quad.
\]
\end{lemma}

\begin{proof}
This is a formal consequence of~\eqref{eq:shad}:
\begin{align*}
\varphi_*(\uC') &= \varphi_*\left(c(\varphi^*\cE)\cap 
\beta_*(c(\cO(-1))^{-1}\cap C')\right) \\
&= c(\cE)\cap \varphi_*\beta_*(c(\cO(-1))^{-1}\cap C') \\
&= c(\cE)\cap \alpha_*\hat\varphi_*(c(\cO(-1))^{-1}\cap C') \\
&= c(\cE)\cap \alpha_*(c(\cO(-1))^{-1}\cap \hat\varphi_*(C')) \\
&= \underline{\hat\varphi_*(C')}
\end{align*}
as stated.
\end{proof}

\begin{lemma}\label{lem:shapb}
Assume $\varphi,\hat\varphi$ are flat, and let $C$ be a class in
$\Pbb(\cE)$. Then
\[
\underline{\hat\varphi^*(C)} = \varphi^* \left( \uC \right)\quad.
\]
\end{lemma}

\begin{proof}
This is also formal, using~\eqref{eq:shad}:
\begin{align*}
\underline{\hat\varphi^*(C)} &= c(\varphi^*\cE)\cap 
\beta_*\left( c(\cO(-1))^{-1}\cap \hat\varphi^*(C) \right) \\
&= c(\varphi^*\cE)\cap \beta_*\hat\varphi^* 
\left(c(\cO(-1))^{-1}\cap C \right) \\
&= c(\varphi^*\cE)\cap \varphi^*\alpha_* \left(c(\cO(-1))^{-1}\cap C \right)
\intertext{since the diagram is a fiber square}
&= \varphi^* \left( c(\cE)\cap \alpha_* (c(\cO(-1))^{-1}\cap C)\right) \\
&= \varphi^* (\uC)
\end{align*}
as claimed.
\end{proof}

Both Lemma~\ref{lem:shapf} and~\ref{lem:shapb} are straightforward
consequences of the basic properties of push-forwards and pull-backs of
regular morphisms.  We are interested in a different type of
functoriality, involving {\em rational\/} morphisms, and which will
require a little more work.

Let $\rho:V\dashrightarrow W$ be a dominant rational morphism of
varieties, and let $S$ be a subscheme of $V$. We say that $p$ is a
{\em projection with center at $S$\/} if the blow-up $\pi: \Til V \to
V$ of $V$ at $S$ resolves the indeterminacies of $\rho$ and the lift
$\tilde \rho: \Til V \to W$ is flat.
\[
\xymatrix{
& \Til V \ar[dl]_\pi \ar[dr]^\trho \\
V \ar@{-->}[rr]^\rho & & W
}
\]
If $\rho:V\dashrightarrow W$ is a projection with center at $S$, and 
$\gamma\in A_*W$, we define the {\em join\/} of $\gamma$ and 
$S$ to be the class
\[
\gamma \vee S:=\pi_*\tilde\rho^*(\gamma) \in A_*V\quad.
\]
If $S=\emptyset$, then $\rho$ is a flat regular map, and the join
operation is the ordinary pull-back. The main result of this section
is the extension of Lemma~\ref{lem:shapb} to the case of projections.

We consider vector bundles $\cE_V$, resp., $\cE_W$ of the same rank
$r+1$ on $V$, resp.~$W$.

\begin{defin}\label{def:compa}
We say that $\cE_V$, $\cE_W$ are {\em compatible\/} if their pull-backs
to $V\smallsetminus S$ are isomorphic:
\[
\cE_V|_{V\smallsetminus S} \cong {\rho'}^*\cE_W
\]
where $\rho': V\smallsetminus S \to W$ is the restriction of
$\rho$. If $\cE_V$ and $\cE_W$ are compatible, we choose an
isomorphism $\Pbb(\cE_V|_{V\smallsetminus S})\cong
\Pbb({\rho'}^*\cE_W)$ and use it to identify the corresponding Chow
groups. Choices of classes $C_V\in A_d \Pbb(\cE_V)$, $C_W\in A_{d-\dim
  V+\dim W} \cE_W$ are then {\em compatible\/} if they agree after
pull-back to the restrictions over $V\smallsetminus S$.  
\qede\end{defin}

\begin{theorem}\label{thm:join}
Let $\rho: V\dashrightarrow W$ be a projection with center at $S$, and
let $\cE_V$, resp., $\cE_W$ be compatible vector bundles of rank
$(r+1)$ on $V$, resp., $W$.  Let $C_V\in A_d(\Pbb(\cE_V))$, $C_W\in
A_{d-\dim V+\dim W}(\Pbb(\cE_W))$ be compatible classes, with
$d>r+\dim S$. Then
\[
\underline{C_V} = \underline{C_W}\vee S\quad.
\]
\end{theorem}

\begin{proof}
The notation we need may be found in the following diagram.
\[
\xymatrix@C=10pt{
& \Pbb(\pi^* \cE_V) \ar[dr] \ar[dl]_{\hat\pi} & & \Pbb(\trho^*\cE_W) 
\ar[dl] \ar[dr]^{\hat\rho} \\
\Pbb(\cE_V) \ar[dr] & & \Til V \ar[dl]_\pi \ar[dr]^\trho & & \Pbb(\cE_W) 
\ar[dl] \\
& V \ar@{-->}[rr]^\rho & & W
}
\]
The compatibility of $\cE_V$, $\cE_W$ implies that the pull-backs of
$\cE_V$, $\cE_W$ to the complement $V^\circ$ of the exceptional
divisor $E$ in $\Til V$ are isomorphic:
\[
i^* \pi^* \cE_V \cong i^*\trho^* \cE_W
\]
where $i: V^\circ \hookrightarrow \Til V$ is the open embedding. The
compatibility of $C_V$ and $C_W$ implies that the restrictions of
$\hat\pi^*(C_V)$ and $\hat\rho^*(C_W)$ to $V^\circ$ coincide. By
Lemma~\ref{lem:shapb}, the corresponding shadows coincide after
restriction to $V^\circ$:
\[
i^* \underline{\hat\pi^*(C_V)} = \underline{i^*\hat\pi^*(C_V)} 
=\underline{i^*\hat\rho^*(C_W)}=i^* \underline{\hat\rho^*(C_W)}\quad.
\]
By the exact sequence of Chow groups for an open embedding
(\cite[\S1.8]{85k:14004}), there exists a class $\gamma$ in $A_*(E)$
such that
\begin{equation}\label{eq:upeq}
\underline{\hat\pi^*(C_V)} = \underline{\hat\rho^*(C_W)} + j_*(\gamma)
\end{equation}
in $A_*(\Til V)$, where $j: E\hookrightarrow \Til V$ is the
inclusion. Note that nonzero components of $j_*(\gamma)$ necessarily
have dimension $\ge d-r$, since this is the case for the other classes 
appearing in~\eqref{eq:upeq}.

Since $\trho$ is flat, we have $\underline{\hat\rho^*(C_W)} =
\trho^*(\underline{C_W})$, also by Lemma~\ref{lem:shapb}. On the other
hand, $\pi_*(\underline{\hat\pi^*(C_V)}) =\underline{C_V}$ by
Lemma~\ref{lem:shapf}. Therefore, \eqref{eq:upeq} implies
\begin{equation}\label{eq:eqdown}
\underline{C_V}=\pi_* \trho^*(\underline{C_W}) + \pi_* j_*(\gamma)
\end{equation}
after pushing forward by $\pi$. By definition,
$\pi_*\trho^*(\underline{C_W}) = \underline{C_W}\vee S$. Finally,
$\pi_* j_*(\gamma)$ can only have nonzero components in dimension $\ge
d-r$, but it is supported on $S$, whose dimension is $<d-r$ by
hypothesis. Therefore $\pi_* j_*(\gamma)=0$, the right-hand side
of~\eqref{eq:eqdown} equals $\underline{C_W}\vee S$, and the stated
equality follows.
\end{proof}

If $S=\emptyset$, then the hypothesis $d>r+\dim S$ is vacuous, $\rho$
is flat, $\rho^*(\cE_W)\cong \cE_V$, $\underline{C_W}\vee
S=\rho^*(\underline{C_W})$, and Theorem~\ref{thm:join} reduces to
Lemma~\ref{lem:shapb}.  If $S\neq \emptyset$, then the hypothesis
$d>r+\dim S$ is necessary.

While we will not need this for our main result, it is occasionally 
useful to relate the shadow of a class $C\in \Pbb(\cE)$ to the 
shadow of the class $C_\cL$ corresponding to $C$ in 
$\Pbb(\cE\otimes \cL)$, where $\cL$ is a line bundle on $V$.  
The reader should have no difficulties proving the following 
statement, which uses the notation introduced
in~\cite[\S2]{MR96d:14004}.

\begin{lemma}\label{lem:twsha}
Let $C\in A_d\Pbb(\cE)\cong A_d\Pbb(\cE\otimes \cL)$,
and let $\uC$, $\uC_\cL$ be respectively the shadows of $C$ viewed as
a class in $\Pbb(\cE)$ and as a class in $\Pbb(\cE\otimes \cL)$. Then
\[
\uC_\cL = c(\cL)^{\rk \cE-1+\dim V-d}\cap (\uC\otimes_V \cL)\quad.
\]
\end{lemma}


\section{Shadows and Segre classes}\label{sec:shaseg}

As in~\S\ref{sec:funsha}, let $V$ be a variety. Let $Z\subsetneq V$ be
a proper subscheme, realized as the zero-scheme of a section $\sigma$
of a vector bundle $\cE$ on $V$. The blow-up of $V$ along $Z$ may be
embedded in the projectivization of $\cE$: $\sigma$ induces a {\em
  rational\/} section $\sigma: V \dashrightarrow \Pbb(\cE)$, and
$B\ell_ZV$ may be identified with the (closure of the) image of this
section.
We may then consider the shadow $\uBl$ of the class of $B\ell_ZV$ in
$\Pbb\cE$; this is a class in $A_*V$. There is a very straightforward
relation between $\uBl$ and the push-forward to $V$ of the segre class
$s(Z,V)$.

\begin{lemma}\label{lem:sfroms}
Let $\iota:Z \hookrightarrow V$ be the embedding. Then
\[
\iota_* s(Z,V)=[V] - s(\cE)\cap \uBl\quad.
\]
\end{lemma}

\begin{proof}
Let $\alpha:\Pbb(\cE) \to V$ be the projection. Using \eqref{eq:shad},
\begin{align*}
[V]-s(\cE)\cap \uBl &= [V]- s(\cE) c(\cE) \alpha_*(c(\cO(-1))^{-1}\cap
[B\ell_ZV]) \\
&=[V]- \alpha_*(c(\cO(-1))^{-1}\cap [B\ell_ZV]) \\
&=\alpha_*\left([B\ell_ZV] - c(\cO(-1))^{-1}\cap [B\ell_ZV]\right)\quad.
\end{align*}
The class $c(\cO(-1))$ restricts to $1+E$ on $B\ell_ZV$, where $E$ is
the exceptional divisor. Therefore
\[
[V]-s(\cE)\cap \uBl 
=\alpha_*\left(\left(1-\frac 1{1+E}\right)\cap [B\ell_ZV]\right)
=\alpha_*\left(\frac {[E]}{1+E}\right)\quad.
\]
The statement follows (by \cite[Corollary~4.2.2]{85k:14004}).
\end{proof}

\begin{remark}
(i) We could restrict $\cE$ to $Z$, and consider the shadow $\uE$ of
  the exceptional divisor $E$ as a class in $A_*Z$. Then
  $s(Z,V)=s(\cE)\cap \uE$, as the reader may verify.

(ii) If $V$ is nonsingular, and $Z$ is the zero-scheme of a section of
  the tangent bundle $TV$, then $\uE$ equals the Chern-Fulton class
  $\cf(Z)$ of $Z$, defined in~\cite[Example~4.2.6(a)]{85k:14004}
  (in particular, it is independent of $V$).

(iii) In $V$, $\uBl+\uE$ is a decomposition of the total Chern class
  $c(\cE)\cap [V]$ as a sum of two classes. This decomposition
  generalizes the decomposition of $c(\cL)$ as $1+D$, where $D$ is the
  divisor determined by a section of a line bundle $\cL$.
\qede\end{remark}

\begin{example}
Let $V=\Pbb^n$, and assume that the ideal of $Z$ is generated by $r+1$
forms of degree~$d$: that is, $\cE=\cO(d)^{r+1}$. Applying
Lemma~\ref{lem:sfroms} gives
\[
\iota_* s(Z,\Pbb^n) = [\Pbb^n]-(1+dH)^{-r-1}\cap \uBl_{\cO(d)}
\]
where $\uBl_{\cO(d)}$ is the shadow of $B\ell_Z\Pbb^n$ as a class in
$\Pbb(\cE)$. On the other hand, tensoring $\cE$ by $\cO(-d)$ realizes
$\Pbb(\cE)$ as a trivial bundle $\Pbb^n\times \Pbb^r$. According to
Lemma~\ref{lem:twsha}, the corresponding shadow changes as follows:
\[
\uBl_{\cO(d)}= (1+dH)^r\cap (\uBl\otimes \cO(dH))\quad,
\]
where $\uBl$ is the shadow with respect to the trivial bundle.
Therefore,
\begin{align*}
\iota_* s(Z,\Pbb^n) &= [\Pbb^n]-(1+dH)^{-r-1}\cap \left(
(1+dH)^r\cap (\uBl\otimes \cO(dH))\right) \\
&= [\Pbb^n]-(1+dH)^{-1}\cap \left( \uBl\otimes \cO(dH)\right)\quad.
\end{align*}
This reproduces Proposition~3.1 in~\cite{MR1956868}.
\qede\end{example}

In view of Lemma~\ref{lem:sfroms}, the functoriality properties proved
in~\S\ref{sec:funsha} imply analogous properties for Segre classes, at
least after push-forward to the ambient variety. Lemmas~\ref{lem:shapf} 
and~\ref{lem:shapb} simply specialize to the good behavior of Segre
classes with respect to proper and to flat morphisms,
\cite[Proposition~4.2]{85k:14004}.  Theorem~\ref{thm:join} determines
the behavior of Segre classes under `projections', as follows.

As in~\S\ref{sec:funsha}, we consider a projection $\rho:
V\dashrightarrow W$ with center in $S\subseteq V$. We consider proper
subschemes $Z\subsetneq W$, $\hZ\subsetneq V$ such that
${\rho'}^{-1}(Z)$ agrees with $\hZ\cap (V\smallsetminus S)$, where
$\rho': V\smallsetminus S\to W$ is the (regular) restriction of
$\rho$. More precisely, we assume that we have compatible vector
bundles $\cE_V$ on $V$, $\cE_W$ on $W$, and sections $s_V$ of $\cE_V$,
$s_W$ of $\cE_W$, such that the outer diagram in
\[
\xymatrix@C=10pt{
{\cE_V}|_{V\smallsetminus S} \ar[dr] \ar@{=}[rr]^\sim & & {\rho'}^* \cE_W \ar[dl] \\
& V\smallsetminus S \ar@/^1pc/[ul]^{s_V|_{V\smallsetminus S}} 
\ar@/_1pc/[ur]_{{\rho'}^*s_W}
}
\]
commutes; and we let $Z$, resp., $\hZ$ be the zero-schemes of $s_W$,
resp., $s_V$.  Loosely speaking, $\hZ$ is the inverse image of $Z$
under $\rho$ (but of course an inverse image is not defined as $\rho$
is only assumed to be rational).

\begin{theorem}\label{thm:segcla}
Let $Z\overset i\hookrightarrow W$, resp.~$\hZ\overset
\hi\hookrightarrow V$ be zero-schemes of matching sections of
compatible vector bundles of rank $r+1$, as above. Assume $r<(\dim
V-\dim S)$. Then
\begin{equation}\label{eq:mainid}
\hi_* s(\hZ,V) = [V]-s(\cE_V)\cap \left( (c(\cE_W)
\cap ([W]-i_* s(Z,W)))\vee S\right)\quad.
\end{equation}
\end{theorem}

\begin{proof}
After projectivizing the bundles, we have the commutative diagram
\[
\xymatrix@C=10pt{
\Pbb({\cE_V}|_{V\smallsetminus S}) \ar@{=}[rr]^\sim & & 
\Pbb({\rho'}^* \cE_W) \\
& V\smallsetminus S \ar@/^1pc/@{-->}[ul]^{s_V|_{V\smallsetminus S}} 
 \ar@/_1pc/@{-->}[ur]_{{\rho'}^*s_W}
}
\]
where $s_V|_{V\smallsetminus S}$, ${\rho'}^*s_W$ are the rational
sections induced by their regular namesakes. It follows that
\[
\overline{s_V(V\smallsetminus S)} 
= \overline{{\rho'}^*s_W(V\smallsetminus S)}\quad,
\]
and this implies that $\left[\overline{s_V(V)}\right]$ and 
$\left[\overline{s_W(W)}\right]$
are compatible classes in the sense of
Definition~\ref{def:compa}. These are the classes of the blow-ups
$[B\ell_\hZ V]$ and $[B\ell_ZW]$, respectively. Also, the dimension
$d$ of these classes is $\dim V$, and $\dim V>r+\dim S$ by
hypothesis. By Theorem~\ref{thm:join}, we have an equality of shadows
\[
\underline{[B\ell_\hZ V]} = \underline{[B\ell_Z W]} \vee S\quad.
\] 
The stated formula is then a direct consequence of Lemma~\ref{lem:sfroms}.
\end{proof}

\begin{corol}\label{cor:stronid}
With the same notation, assume $r+1<(\dim V-\dim S)$. Then
\begin{equation}\label{eq:strin}
\hi_* s(\hZ,V) = s(\cE_V)\cap \left( (c(\cE_W)\cap 
i_* s(Z,W))\vee S\right)\quad.
\end{equation}
\end{corol}

\begin{proof}
For all $i$, we have the exact sequence
\[
\xymatrix{
A_iS \ar[r] & A_i V \ar[r] & A_i(V\smallsetminus S) \ar[r] & 0
} 
\]
(\cite[Proposition~1.8]{85k:14004}). It follows that $A_iV \cong
A_i(V\smallsetminus S)$ for $i>\dim S$.
The fact that the bundles are compatible implies that
$c(\cE_V)\cap [V]$ and $(c(\cE_W)\cap [W])\vee S$ agree after
restriction to $V\smallsetminus S$.
If $r+1<(\dim V-\dim S)$, then the codimension of $S$ exceeds 
the rank of these bundles, hence we can deduce that
$c(\cE_V)\cap [V]=(c(\cE_W)\cap [W])\vee S$ in $A_*V$.

Therefore $s(\cE)\cap ((c(\cE_W)\cap [W])\vee S) = [V]$
if $r+1<(\dim V-\dim S)$, and
the stated equality follows from Theorem~\ref{thm:segcla}.
\end{proof}

\begin{remark}\label{rem:stronid}
A refinement of the argument proving Theorem~\ref{thm:segcla}
shows that~\eqref{eq:strin} holds as an identity of classes {\em on $\hZ$\/}
if $r+1<(\dim V-\dim S)$. 
In our application in~\S\ref{sec:szf} we will only need the equality after 
push-forward, as given in Corollary~\ref{cor:stronid}. 

Identity~\eqref{eq:mainid} cannot be stated in $A_*\hZ$, so this strengthening
is not available for Theorem~\ref{thm:segcla}. However, the advantage of the
weaker requirement on $\dim S$ in the hypotheses of Theorem~\ref{thm:segcla}
may be important from the computational viewpoint.
\qede\end{remark}


\section{The Segre zeta function of a homogeneous ideal}\label{sec:szf}

In this section we define the `Segre zeta function' of a homogeneous
ideal in a polynomial ring, and prove its rationality and other
features.  We will apply the results obtained in the previous sections
to the projection $\rho:\Pbb^N \dashrightarrow \Pbb^n$ with center at
a subspace $\Pbb^m$, $m=N-n-1$; this is a projection in the sense used
in the previous sections. The join operation $\gamma\mapsto \gamma
\vee \Pbb^m$ acts as a `partial pull-back' on the Chow ring:
\[
A^i \Pbb^n \to A^i \Pbb^N\quad, \quad \gamma\mapsto \gamma \vee \Pbb^m
\]
is defined by sending the generator $H^i$ of $A^i \Pbb^n$ to the
generator $H^i\in A^i \Pbb^N$; here $H$ denotes the hyperplane class
(in both $\Pbb^n$ and $\Pbb^N$).

Notice that, geometrically, this is indeed a `join' operation: $H^i$
is represented by a linearly embedded $\Pbb^{n-i}$ in $\Pbb^n$; the
corresponding class in $A_{N-i}\Pbb^N$ is the class of the join of
$\Pbb^{n-i}$ and the center $\Pbb^m$ of the projection.

Also note that this is {\em not\/} a ring homomorphism: $H^{n+1}=0$ in
$A^*\Pbb^n$, while $H^{n+1}\neq 0$ in $A^*\Pbb^N$ for $N>n$.  

It will
be convenient to adopt a `cohomological' notation, and represent
classes in projective space $\Pbb^r$ as polynomials of degrees $\le r$
in the hyperplane class, which will uniformly be denoted $H$. With
this convention, the join operation $A_*\Pbb^n \to A_*\Pbb^N$ acts in
the simplest possible way:
\[
\gamma \mapsto \gamma \vee \Pbb^m\quad:\quad P(H) \mapsto P(H)\quad.
\]
However, care has to be taken to ensure that the polynomial at the source has
degree $\le n$. We will denote by $[P(H) ]_{\le n}$ the truncation of
the polynomial $P(H)$ to $H^n$. This operation may be extended to
power series in $H$.

\begin{example}
With this notation,
\[
\left[\frac {dH}{1+dH}\right]_4 = dH-d^2H^2+d^3H^3-d^4H^4
\]
represents $s(X,\Pbb^4)\in A_*\Pbb^4$ for a degree-$d$ hypersurface
$X$ in $\Pbb^4$. It also represents the join $s(X,\Pbb^4)\vee
\Pbb^m\in A_* \Pbb^{m+5}$, for every $m\ge 0$.
\qede\end{example}

Now let $I\subseteq k[x_0,\dots, x_n]$ be a homogeneous
ideal. For any $N>n$, we let $I_N\subseteq k[x_0,\dots, x_N]$ be
the extension of $I$, and we denote by 
$Z_N\overset{\iota_N}\hookrightarrow\Pbb^N$ the subscheme defined
by~$I_N$.  As above, we will denote by $H$ the hyperplane class (in
any projective space).

\begin{lemma}\label{lem:wd}
There exists a well-defined power series
\[
\zeta_I(t) = \sum_{i\ge 0} a_i t^i \in \Zbb[[t]]
\]
such that for all $N\ge n$, 
\[
\iota_{N*}s(Z_N,\Pbb^N) = \sum_{i=0}^N a_i H^i \cap [\Pbb^N]\quad.
\]
\end{lemma}

\begin{defin}
The {\em Segre zeta function\/} of the ideal $I$ is the power series
$\zeta_I(t)$ obtained in Lemma~\ref{lem:wd}.
\qede\end{defin}

\begin{proof}[Proof of Lemma~\ref{lem:wd}]
We have to show that if $M>N$ and
\[
\iota_{N*} s(Z_N,\Pbb^N) = \sum_{i=0}^N a_i [\Pbb^{N-i}] \quad, \quad
\iota_{M*} s(Z_M,\Pbb^M) = \sum_{i=0}^M b_i [\Pbb^{M-i}] \quad, 
\]
then $b_i = a_i$ for $i=0,\dots, N$.

In this situation we can view $\Pbb^N$ as embedded in $\Pbb^M$ as the
linear subspace defined by $x_{N+1}=\cdots = x_M=0$. We have
\[
I_N = I_M \cap k[x_0,\dots, x_N]\quad,
\]
and correspondingly $Z_N = \Pbb^N \cap Z_M$. This intersection is
transversal (in fact, splayed), hence repeated application of
\cite[Lemma~4.1]{MR3415650} gives
\begin{equation}\label{eq:tranre}
s(\hZ_N,\Pbb^N) =[\Pbb^N]\cdot s(\hZ_M,\Pbb^M)\quad.
\end{equation}
Therefore
\[
\sum_{i=0}^N a_i [\Pbb^{N-i}]  = [\Pbb^N] \cdot 
\sum_{i=0}^M b_i [\Pbb^{M-i}]
\]
with the stated consequence.
\end{proof}

\begin{remark}\label{rem:less}
By essentially the same argument, we see that if $N<n$, then 
$\sum_{i=0}^N a_i H^i\cap [\Pbb^N]$ equals 
$\iota_{N*} s(Z_N,\Pbb^N)$, where now $Z_N$ is the intersection
of $Z_n$ with a general $N$-dimensional linear subspace of 
$\Pbb^n$. In particular, the first nonzero coefficient $a_s$
in $\zeta_I(t)$ occurs for $s=$ the codimension of $Z_N$
in $\Pbb^N$ for $N$ large enough; it follows that $s=\codim I$. 
Also, $a_s$ equals the degree of the top-dimensional part of 
the cycle determined by $[Z_N]$. For example, if $I$ is prime, then
\[
\zeta_I(t) = (\deg I)\, t^{\codim I} + \text{higher order terms}\quad.
\]
where $\deg I=\deg Z_N$ (for $N\gg 0$).
\qede\end{remark}

\begin{example}\label{ex:compint}
If $I$ is a complete intersection, i.e., it is generated by a regular sequence 
$(F_1,\dots, F_r)$, with $\deg F_i=d_i$, then
\[
\zeta_I(t) = \frac{d_1\cdots d_r\, t^r}{(1+d_1 t)\cdots (1+d_r t)}\quad.
\]
Indeed, $Z_N$ is then regularly embedded in $\Pbb^N$, hence its
Segre class is given by $s(Z_N,\Pbb^N) = c(N_{Z_N}\Pbb^N)^{-1}
\cap [Z_N]$ (\cite[Proposition~4.1]{85k:14004}). The normal
bundle of $Z_N$ has Chern class $(1+d_1 H)\cdots (1+d_r H)$, and
$[Z]=d_1\cdots d_r H^r$.
\qede\end{example}

\begin{example}\label{ex:notsch}
For $I\subseteq k[x_0,\dots, x_n]$, $\zeta_I(t)$ is not determined by 
the scheme $Z$ defined by $I$ in $\Pbb^n$. For example,
consider $I_1=(x_0,\dots, x_n)$ and $I_2=(x_0^2,x_1,\dots, x_n)$.
These ideals both define the empty set in $\Pbb^n$, while
\[
\zeta_{I_1}(t) = \frac{t^{n+1}}{(1+t)^{n+1}}\quad, \quad
\zeta_{I_2}(t) = \frac{2\,t^{n+1}}{(1+2t)(1+t)^n}
\]
as seen in Example~\ref{ex:compint}. We will prove that $\zeta_I(t)$
is determined by $Z$ and by the degrees of a set of generators for
$I$, provided the number of generators does not exceed $n+1$, 
see Corollary~\ref{cor:ngr}; also see Corollary~\ref{cor:ngr2} for 
a somewhat stronger statement.
\qede\end{example}

\begin{example}\label{ex:larger}
Let $I=(x_0^2-x_1^2,x_0x_1^2-x_2^3,x_0^4-x_2^4)$. Using the Macaulay2
implementation of the algorithm in~\cite{MR1956868}, we can compute
the Segre class of the scheme cut out by the generators of $I$ in
$\Pbb^9$, and this determines the first several coefficients of
$\zeta_I(t)$:
\[
\zeta_I(t) = 2t^2+6t^3-106t^4+750t^5-4138t^6+20286 t^7-92986t^8
+408750 t^9-\cdots
\]
(this computation takes several minutes on a Macbook Pro).
\qede\end{example}

The natural question is whether general statements can be made
concerning $\zeta_I(t)$. For example, it is perhaps not completely
obvious from the definition that the power series $\zeta_I(t)$ has
positive radius of convergence.  Statements restricting the type of
series $\zeta_I(t)$ can be are potentially useful in computations of 
Segre classes, as we will illustrate below (Example~\ref{ex:revisit}). 
The main result of this paper is 
the following theorem. 

\begin{theorem}\label{thm:main}
Let $I\subseteq k[x_0,\dots, x_n]$ be a homogeneous ideal, and let
$\zeta_I(t)$ be the power series defined above. Also, let
$d_0, \dots, d_r$ be the degrees of the elements in any
homogeneous basis of $I$. Then
\begin{enumerate}
\item\label{pt:rat} $\zeta_I(t)$ is rational: there exist unique relatively
prime polynomials $P_I(t),Q_I(t)\in \Zbb[t]$, with $Q_I(t)$ monic,
such that
\[
\zeta_I(t) = \frac{P_I(t)}{Q_I(t)}\quad.
\]
\item The polynomial $Q_I(t)$ divides
  $(1+d_0t)\cdots (1+d_r t)$. 
  Thus, the poles of $\zeta_I(t)$ can only occur at $-1/d_i$,
  where $d_i$ is a degree of an element in a minimal homogeneous
  basis of $I$.
\item With notation as above, $\zeta_I(t)=\dfrac{N(t)
  +(\prod_i d_i) t^{r+1} }{\prod_i (1+d_it)}$ 
  for a degree $r$ polynomial $N(t)$ with nonnegative coefficients
  and trailing term of degree~$\codim I$.
\end{enumerate}
\end{theorem}

Part (3) has the following immediate consequence 
(cf.~Example~\ref{ex:notsch}).

\begin{corol}\label{cor:ngr}
The Segre zeta function $\zeta_I(t)$ of an ideal $I$ generated by
homogeneous polynomials $f_0,\dots, f_r\in k[x_0,\dots, x_n]$
is determined by the Segre class of the zero-scheme $Z=Z_n$ 
of $I$ in $\Pbb^n$ and by the integers $\deg f_i$, provided $n\ge r$.
\end{corol}

\begin{remark}
More precisely, if $n\ge r$, then the components $s(Z,\Pbb^n)_j$ with 
$j\ge n-r$ (together with the integers $\deg f_i$) suffice to determine 
$\zeta_I(t)$, since they suffice to determine the polynomial $N(t)$ in
Theorem~\ref{thm:main}~(3).

In fact, if $n>r$, then we can replace $Z$ by the intersection of $Z$
with a general $\Pbb^r$ and reduce to the case $n=r$.
This follows from~\cite[Lemma~4.1]{MR3415650}, which shows
that this operation does not change the terms of codimension
$\le r$ in the Segre class, see~\eqref{eq:tranre}.
(Cf.~Remark~\ref{rem:less}.)
\qede\end{remark}

\begin{example}\label{ex:revisit}
As an illustration of how Theorem~\ref{thm:main} may be useful in 
computations, revisit Example~\ref{ex:larger}, taking 
Corollary~\ref{cor:ngr} into account. 
Here $n=r=2$, $d_0=2$, $d_1=3$, $d_2=4$.  
By Theorem~\ref{thm:main} (3),
\[
\zeta_I(t) = \frac{N(t)+24t^3}{(1+2t)(1+3t)(1+4t)}\quad,
\]
with $\deg N(t)\le 2$. It follows that $\zeta_I(t)$ is determined by
the coefficients of the terms of degree $\le 2$, i.e., by the Segre
class of the scheme cut out by the generators of $I$ in $\Pbb^2$. The
same implementation of the algorithm used in Example~\ref{ex:larger}
computes this information,
\[
\zeta_I(t) \equiv 2t^2 \mod{t^3}\quad,
\]
in less than half a second. It follows that
\[
N(t) = [(2t^2+6t^3)(1+2t)(1+3t)(1+4t)]_2 = 2t^2
\]
and therefore the same data obtained above,
\begin{align*}
\zeta_I(t) &= \frac{2t^2+24t^3}{(1+2t)(1+3t)(1+4t)} \\ 
&=2t^2+6t^3-106t^4+750t^5-4138t^6+20286 t^7-92986t^8+408750 t^9-\cdots\quad,
\end{align*}
is found with a roughly 1000-fold increase in speed with respect to the
more direct computation. (And $\zeta_I(t)$ is now known to all
orders.)  
\qede\end{example}

Theorem~\ref{thm:main} admits the following refinement, which also
leads to a strengthening of Corollary~\ref{cor:ngr}.
Recall that the {\em degree sequence\/} of an ideal is the sequence 
$d_0\le \cdots \le d_r$ of degrees of elements of any homogeneous 
minimal basis for the ideal; this sequence does not depend on the 
chosen minimal basis. We also recall that a {\em reduction\/} of an 
ideal $I$ is an ideal $J\subseteq I$ with the same integral closure as $I$.
We denote by $\overline I$ the integral closure of $I$.

\begin{prop}\label{pro:redux}
$\zeta_{\overline I}(t)=\zeta_I(t)$.
\end{prop}

\begin{corol}\label{cor:ref}
The results of parts (2) and (3) of Theorem~\ref{thm:main} hold with
$d_0\le \cdots\le d_r$ the degree sequence of any homogeneous
reduction of $I$:
\begin{itemize}
\item $Q_I(t)\,|\, \prod_i (1+d_i t)$;
\item $\zeta_I(t)=\dfrac {N(t) +(\prod_i d_i) t^{r+1} }{\prod_i (1+d_it)}$ 
  for a degree $r$ polynomial $N(t)$ with nonnegative coefficients
  and trailing term of degree $\codim I$.
\end{itemize}
\end{corol}

Corollary~\ref{cor:ref} is an immediate consequence of 
Theorem~\ref{thm:main} and Proposition~\ref{pro:redux}. We also
highlight the following consequence.

\begin{corol}\label{cor:poles}
Every pole of the Segre zeta function of $I$ is an element of the
degree sequence of a minimal homogeneous reduction of $I$.
\end{corol}

\begin{proof}[Proof of Proposition~\ref{pro:redux}]
The integral closure of the extension of $I$ in $k[x_0,\dots, x_N]$ 
equals the extension of the integral closure,
so it suffices to verify that if $Z$, resp., $\overline Z$ are the schemes 
defined by $I$, resp.~$\overline I$ in $\Pbb^n$, then $s(Z,\Pbb^n)
=s(\overline Z,\Pbb^n)$. (Since $Z$, $\overline Z$ have the same
support, their Chow groups may be identified.) By 
\cite[Proposition~1.44]{MR2153889}, there is a natural finite morphism 
of blow-ups $B\ell_{\overline Z}\Pbb^n \to B\ell_Z \Pbb^n$. It follows
that the inverse image of $Z$ in $B\ell_{\overline Z}\Pbb^n$ equals
the exceptional divisor of $B\ell_{\overline Z}\Pbb^n$,
and the equality follows by the birational invariance of Segre classes
(\cite[Proposition~4.2]{85k:14004}).
\end{proof}

Corollary~\ref{cor:ref} could also be helpful in the construction of
algorithms computing Segre classes: an ideal may be replaced with 
a minimal homogeneous reduction without affecting the computation 
of the Segre zeta function, and in general this reduces the length of 
the degree sequence used to construct a denominator for 
$\zeta_I(t)$. We formalize this observation as follows.

\begin{corol}\label{cor:ngr2}
Let $I\subseteq k[x_0,\dots, x_n]$ be an ideal, and let $Z$ be the
scheme defined by $I$ in~$\Pbb^n$. Then the Segre zeta function 
$\zeta_I(t)$ is determined by the degree sequence 
$d_0\le \cdots \le d_r$ of a minimal reduction of $I$ and by 
$\iota_* s(Z,\Pbb^n)_i$, $i\ge n-r$, provided $n\ge r$.
\qede\end{corol}

\smallskip

The main ingredient in the proof of Theorem~\ref{thm:main} will be the
following statement, where we use the notation introduced at the 
beginning of this section.

\begin{prop}\label{pro:shadco}
Let $I=(f_0,\dots, f_r)\subseteq k[x_0,\dots, x_n]$ be an ideal
generated by homogeneous polynomials $f_i$, $i=0,\dots, r$, and let
$d_i=\deg f_i$. For $N\ge n$, let
$Z_N\overset{\iota_N}\hookrightarrow\Pbb^N$ be the subscheme defined
by $f_0,\dots, f_r$ in $\Pbb^N$ (as above).
\begin{equation}\label{eq:nprecis}
\text{If $n\ge r+1$, then}\quad\iota_{N*}s(Z_N,\Pbb^N) 
=\left[ \frac{[\prod_i (1+d_i H)\cdot \iota_{n*} s(Z_n,\Pbb^n)]_n}
{\prod_i (1+d_i H)}\right]_N \quad.
\end{equation}
\end{prop}

Identity~\eqref{eq:nprecis} should be compared with~\eqref{eq:1min}.
As promised in~\S\ref{sec:1min}, the rational function expressed 
by~\eqref{eq:1min} is independent of all choices.

\begin{proof}
We apply Corollary~\ref{cor:stronid} with $W=\Pbb^n$, $V=\Pbb^N$, $\rho:
\Pbb^N \dashrightarrow \Pbb^n$ the projection with center at $\Pbb^m$,
$m=N-n-1$.  We view $I\subseteq k[x_0,\dots, x_n]$ with generators
$f_0,\dots, f_r$ in degree $d_0,\dots, d_r$, as the ideal of the
zero-scheme $Z=Z_n$ of the section $s_W=(f_0,\dots, f_r)$ of
$\cE_W=\cO_{\Pbb^n}(d_0)\oplus\cdots\oplus \cO_{\Pbb^n}(d_r)$. We take
$\cE_V=\cO_{\Pbb^N}(d_0)\oplus\cdots\oplus \cO_{\Pbb^N}(d_r)$; it is
clear that $\cE_W$ and $\cE_V$ are compatible in the sense of
Definition~\ref{def:compa}.  It is also clear that the section
$s_V=(f_0,\dots, f_r)$ {\em of $\cE_V$\/} is compatible with $s_W$,
and $\hZ=Z_N$ is its zero-scheme. We
are therefore in the situation of Corollary~\ref{cor:stronid}, and we can
conclude that
\[
\hi_* s(\hZ,V) = s(\cE_V)\cap \left( (c(\cE_W)\cap i_* s(Z,W))\vee S\right)\quad.
\]
if $r+1<(\dim V-\dim S)=N-m=n+1$, i.e., $n\ge r+1$. This gives~\eqref{eq:nprecis} 
as needed.
\end{proof}

\begin{remark}
The same argument, using Theorem~\ref{thm:segcla}, gives the statement:
\[
\text{If $n\ge r$, then:}\quad\iota_{N*}s(Z_N,\Pbb^N) 
=\left[ 1-\frac{[(\prod_i (1+d_i H)(1-
\iota_{n*}s(Z_n,\Pbb^n)))]_n}{\prod_i (1+d_i H)}\right]_N \quad.
\]
This is stronger than~\eqref{eq:nprecis}, in the sense that 
the advantage of computing the `input' Segre class at numerator 
in $\Pbb^r$, i.e., with $n=r$, rather than $\Pbb^{r+1}$ may be 
substantial. (In our illustrative Examples~\ref{ex:larger} 
and~\ref{ex:revisit}, the computation in $\Pbb^3$ takes 
about twice as long as the computation in $\Pbb^2$.)
This advantage is absorbed by the fact that the term of
degree $r+1$ in the numerator of~\eqref{eq:nprecis}
is in fact known {\em a priori,\/} as stated in 
Theorem~\ref{thm:main}~(3) and as we will show in
a moment.
\qede\end{remark}

Since the numerator of~\eqref{eq:nprecis} is independent of $N$, 
Proposition~\ref{pro:shadco} implies the rationality of $\zeta_I(t)$.
This establishes Theorem~\ref{thm:main}, parts (1) and (2).

To prove part (3), let $f_0,\dots, f_r$ be homogeneous generators 
of $I$, let $d_i=\deg f_i$, and choose $n\ge r+1$. Then 
\[
\zeta_I(H) = \frac{[\prod_i (1+d_i H)\cdot \iota_{n*} s(Z_n,\Pbb^n)]_n}
{\prod_i (1+d_i H)}
\]
by~\eqref{eq:nprecis}. Let $X_i\subseteq \Pbb^n$ be the hypersurface
defined by the vanishing of $f_i$, so that $Z_n=X_0\cap\cdots \cap X_r$. 
The numerator
\begin{equation}\label{eq:numer}
\prod_i (1+d_i H)\cdot \iota_{n*} s(Z_n,\Pbb^n) \in A_*\Pbb^n
\end{equation}
is the push-forward to $\Pbb^n$ of the class used to define the intersection 
product $X_0\cdots X_r$ in~$\Pbb^n$, as recalled in~\S\ref{sec:1min}, by
means of the diagram
\[
\xymatrix{
Z_n \ar[r]^{\iota_n} \ar[d]_\delta & \Pbb^n \ar[d]^\Delta \\
X_0\times \cdots \times X_r \ar[r] & \Pbb^n\times\cdots \times \Pbb^n
}
\]
In particular, its term of degree $r+1$ in $H$ is $d_0\cdots d_r\, H^{r+1}$, 
by B\'ezout's theorem; therefore the coefficient of $t^{r+1}$ in~\eqref{eq:numer} 
equals $\prod_i d_i$, as claimed in Theorem~\ref{thm:main} (3).
The trailing term is discussed in~Remark~\ref{rem:less}.
Further, the class~\eqref{eq:numer} may also be written as follows.
Let $N=\delta^*(N_{X_0\times \cdots \times X_r} \Pbb^n\times \cdots
\times \Pbb^n$). Then the normal cone $C$ of $Z_n$ in $\Pbb^n$
may be embedded in $N$, and~\eqref{eq:numer} equals
\begin{equation}\label{eq:numer2}
\iota_{n*} q_*\left( c(\xi)\cap [\Pbb(C\oplus 1)]\right)
\end{equation}
where $\xi$ is the universal quotient bundle on $\Pbb(N\oplus 1)$ 
and $q$ is the projection from $\Pbb(N\oplus 1)$ to $Z_n$.
(This follows from the projection formula and the definition of
Segre class.) Since $\dim \Pbb(C\oplus 1) = n$ and $\rk \xi = r+1$,
the components of this class of codimension $>(r+1)$ in~$\Pbb^n$
necessarily vanish. This shows that~\eqref{eq:numer}
has degree $(r+1)$ as a polynomial in~$H$. 

The last remaining assertion in Theorem~\ref{thm:main}~(3)
is that the numerator has nonnegative coefficients, that is, that
\eqref{eq:numer} is effective. But note that
\[
N\cong \iota^* \left( \cO(d_0)\oplus \cdots \oplus \cO(d_n)\right)
\]
with $d_i>0$, hence it is generated by global section. As $\xi$ 
is a quotient of $q^*(N\oplus 1)$, it is also generated by global 
sections. It follows that \eqref{eq:numer2} is nonnegative, by
\cite[Example~12.1.7(a)]{85k:14004}, and this concludes the proof
of~Theorem~\ref{thm:main}.
\qed\smallskip

For clarity, we present here the skeleton of an algorithm computing
the Segre zeta function of an ideal, assuming that a basic
algorithm computing Segre classes of subschemes in projective
space is available. As illustrated in Example~\ref{ex:revisit},
this can also act as a bootstrap for current algorithms computing
Segre classes, improving
their performance quite substantially in some cases.
Assume $I$ is an ideal in $k[x_0,\dots, x_n]$, defined by 
homogeneous generators $f_i$.
\begin{itemize}
\item Test each $f_i$ to see if it is in the integral closure of the
other generators. If it is, remove it from the list and replace
$I$ by the ideal generated by the remaining elements. Repeat
until no such element is left.
\item Let $r+1$ be the number of generators left. If $n<r$, 
extend $I$ to ensure $n=r$. If $n>r$, replace $r-n$ variables $x_i$
with general linear combinations of the other $r$ variables,
again to obtain $r=n$, and restrict $I$.
\item Compute $\iota_* s(Z,\Pbb^r)$, where 
$Z\overset\iota\hookrightarrow \Pbb^r$ is the subscheme defined 
by $I$. Write it as a polynomial $S(t)$ of degree $\le r$. Let 
$N(t)=[S(t) \prod_i (1+ (\deg f_i)t)]_r$.
\item The Segre zeta function of the given ideal is then $\zeta_I(t) =
\dfrac{N(t) + (\prod_i \deg f_i) t^{r+1}} {\prod_i (1+ (\deg f_i)t)}$.
\end{itemize}
(Some of these steps are optional: for example, testing for integral 
dependence may be computationally demanding, and it can be
omitted.) Once the Segre zeta function
is known, extracting the information of $s(Z_N,\Pbb^N)$ is immediate
for {\em all\/} $N\ge 0$.

\section{Challenges and examples}\label{sec:cha}

There are several natural problems raised by the description 
of the Segre zeta function obtained in~\S\ref{sec:szf}.\medskip 
 
{\em (1) Determine the poles of $\zeta_I(t)$.\/}

Let $d_0\le \cdots \le d_r$ be the degree sequence of a minimal 
homogeneous reduction of $I$.
As observed in Corollary~\ref{cor:poles},
\begin{equation}\label{eq:mono}
\text{$-1/d$ is a pole of $\zeta_I(t)$} \implies
\text{$d$ is a number in this sequence.}
\end{equation}
However, examples show that $Q_I(t)$ does not necessarily equal 
$\prod_i (1+d_i t)$ (see~\S\ref{ss:mono}). It would be interesting to
examine the extent to which a converse of~\eqref{eq:mono} may 
hold.\smallskip

{\em (2) Describe a numerator of $\zeta_I(t)$ explicitly.\/}

By our main result, if $d_i$, $i=0,\dots,r$ are the degrees of any 
choice of generators for $I$ (or even of a reduction of $I$), then
$(\prod_i (1+d_i t))\,\zeta_I(t)$ is a
polynomial with nonnegative coefficients and leading term 
$d_0\cdots d_r \,t^{r+1}$. The fact that the coefficients are nonnegative
suggests that these may be expressed as dimensions of vector spaces
associated with $I$ (maybe ranks of suitable cohomology modules?),
or perhaps as volumes of polytopes determined by~$I$.\smallskip

{\em (3) Study the behavior of $\zeta_I(t)$ with respect to standard
ideal operations.\/}

For example, assume $I', I''\subseteq k[x_0,\dots, x_n]$ are 
{\em splayed;\/} for instance, we could assume that generators for $I'$ 
and $I''$ are polynomials in different sets of variables. Then
\[
\zeta_{I'+ I''}(t) = \zeta_{I'}(t)\cdot \zeta_{I''}(t)\quad.
\]
Indeed, it suffices to verify that in this case, for $N\gg 0$
\[
s(Z'_N\cap Z''_N,\Pbb^N)=s(Z'_N,\Pbb^N)\cdot s(Z''_N,\Pbb^N)
\]
where $Z'_N$, $Z''_N$ are the schemes defined by extensions
of $I'$, $I''$. Since the extensions are also splayed, this holds
by~\cite[Lemma~3.1]{MR3415650}.\smallskip

{\em (4) Compute projective invariants of a subvariety 
$Z\subseteq \Pbb^N$ in terms of the Segre zeta function of an 
ideal defining $Z$.}\smallskip

Concerning (2), we can describe a numerator for $\zeta_I(t)$ 
more explicitly in two situations, presented in~\S\ref{ss:ls} 
and~\S\ref{ss:mono}. We will illustrate~(4) in~\S\ref{ss:rks}
by computing the {\em ranks\/} of a nonsingular subvariety
$Z\subseteq \Pbb^n$ in terms of its Segre zeta function, and
we show how this implies the well-known fact that the dual
of a nonsingular complete intersection is a hypersurface.
In~\S\ref{sec:HC} we prove that, under certain hypotheses,
the Segre zeta function of a {\em local\/} complete intersection 
in projective space equals the Segre zeta function of a 
{\em global\/} complete intersection.

\subsection{Linear systems}\label{ss:ls}
Assume $Z$ is cut out by hypersurfaces from a fixed linear
system, i.e., the homogeneous generators $f_i$ of 
$I\subseteq k[x_0,\dots, x_n]$ have the same degree $d$.
A reduction of $I$ is then generated by $\le n+1$ general
linear combinations of the polynomials $f_i$.
(In fact, $n$ suffice in a neighborhood of~$Z$, \cite[\S3]{segre};
and the $(n+1)$-st guarantees that the linear combinations do 
not vanish elsewhere.) By Corollary~\ref{cor:ref}, 
\[
\zeta_I(t) = \frac{A(t)}{(1+dt)^{n+1}}
\]
for some polynomial $A(t)$ with nonnegative coefficients.
This argument recovers \cite[Theorem~4.3]{tensored}. 
Assuming that $k$ is algebraically closed, of characteristic $0$,
we obtained in~\cite[(15)]{tensored} the following description 
of $A(t)$:
\[
A(t)=a_0 (1+dH)^n + a_1 H(1+dH)^{n-1} + \cdots 
+ a_n H^n + d^{n+1} t^{n+1}\quad,
\]
where $a_i = d^i - N_i$, $N_i=$ the number of points of
intersection of $i$ general hypersurfaces in the linear system
and $n-i$ general hyperplanes (\cite[Theorem~1.2]{tensored}).

This observation can form the basis of an algorithm computing
the Segre zeta function in this case. However, the saturation 
needed to compute the numbers $N_i$ is computationally
expensive; and it may be argued that one of the main reasons
to compute Segre classes is precisely in order to solve problems
such as determining the numbers $N_i$, so this approach
swims against the stream.
On the other hand, we do not know of any such concrete interpretation
for the numerator of a Segre zeta function in general.

\subsection{Monomial ideals}\label{ss:mono}
Rather than giving the most comprehensive statement, we illustrate
this case with an example. Let
\[
I=(x^7,x^5y,x^4y^2,x^3y^4,x^2y^5,xy^7)\subseteq k[x,y]\quad,
\]
an ideal generated by monomials. The exponent vectors
\[
(7,0)\,,\, (5,1)\,,\, (4,2)\,,\, (3,4)\,,\, (2,5)\,,\, (1,7)
\]
determine a region in the plane, namely the complement in
the positive quadrant of the convex hull of the corresponding
translates of the positive quadrant.
\begin{center}
\includegraphics[scale=.4]{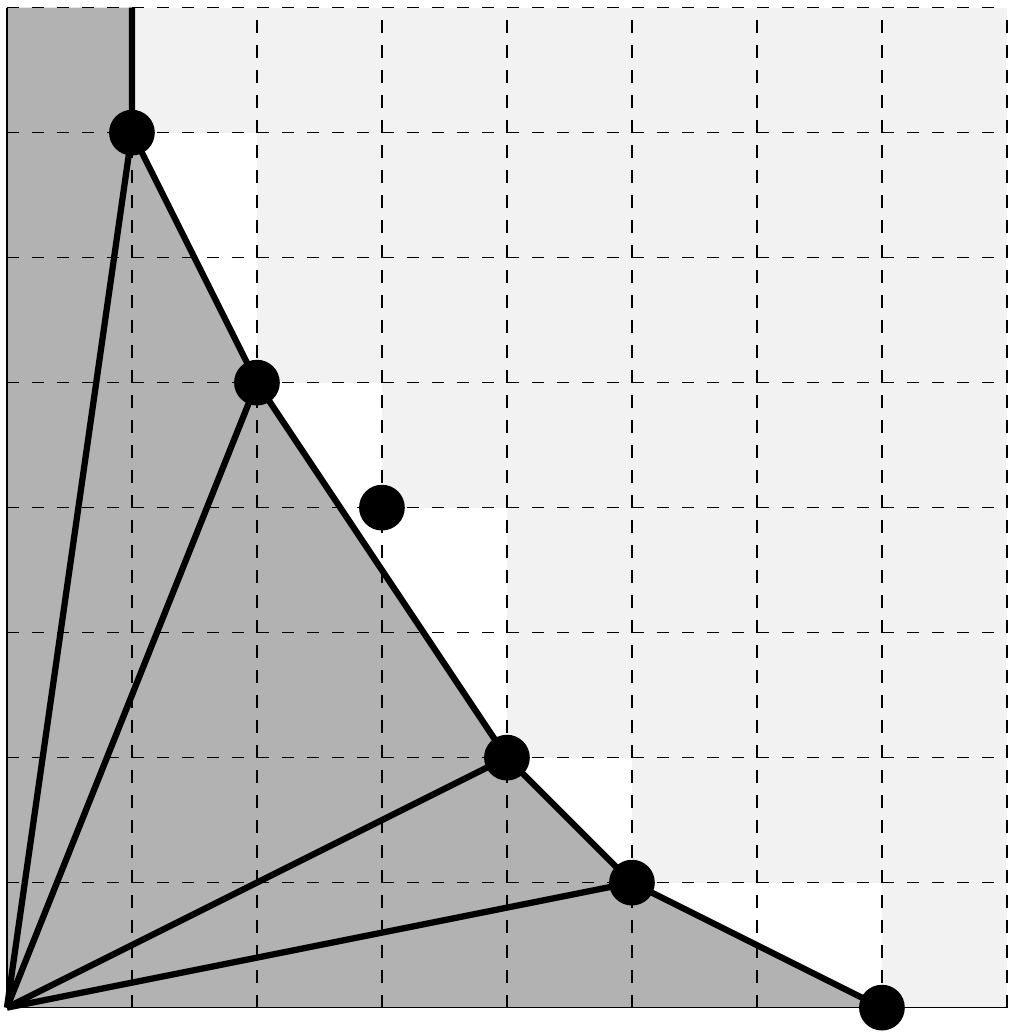}
\end{center}
This region can be split as a union of triangles, including one
`infinite' triangle, as shown in the picture. With each triangle $T$
with vertices $(0,0),(a_1,a_2),(b_1,b_2)$
we associate a rational function:
\[
\frac{\Vol(T)\, t^2}{(1+(a_1+a_2)t)(1+(b_1+b_2)t)}
\]
where $\Vol(T) = |a_1 b_2-a_2 b_1|$ is the normalized volume 
of the triangle. With an infinite triangle with vertices at $(0,0)$,
$(a_1,a_2)$, and the $y$ direction, we associate the rational
function
\[
\frac{\Vol(T)\, t}{(1+(a_1+a_2)t)}
\]
where now $\Vol(T)=a_1$ is the normalized volume of the projection
of the triangle to the $x$-axis. Adding these contributions in the
example shown above, we get
{\small
\[
\frac{t}{1+8t}+\frac{9t^2}{(1+7t)(1+8t)}+\frac{16t^2}{(1+6t)(1+7t)}
+\frac{6t^2}{(1+6t)^2}+\frac{7t^2}{(1+6t)(1+7t)} 
=\frac{t+57t^2+640t^3+2016t^4}{(1+6t)^2(1+7t)(1+8t)}
\]
}

\begin{claim}
\[
\zeta_I(t) = \frac{t+57t^2+640t^3+2016t^4}{(1+6t)^2(1+7t)(1+8t)}
=t+30\, t^2-442\, t^3+4578\, t^4-\cdots
\]
\end{claim}

This follows from~\cite[Theorem~1.1]{Scaiop} (see~\S2.2 in the
same reference for a discussion of the role of triangulations). 
The result of~\cite{Scaiop} shows that the same strategy may be
used to compute the Segre zeta function of any ideal in
$k[x_0,\dots, x_n]$ that is monomial with respect to a sequence
of homogeneous polynomials satisfying a weak transversality
condition. We refer the reader to~\cite{Scaiop} more details.

This example illustrate several interesting points.\smallskip 

$\bullet$ The monomial $x^3y^4$ does not affect the computation,
since the corresponding vertex $(3,4)$ is in the convex hull of
the quadrant translates determined by the other exponent vectors.
Thus it is clear from this computation that $\zeta_I(t)=\zeta_{I'}(t)$,
where
\[
I' = (x^7,x^5y,x^4y^2,x^2y^5,xy^7)
\]
is obtained by omitting the generator $x^3y^4$. This is a manifestation
of Proposition~\ref{pro:redux}: indeed, $x^3 y^4$ is integral over $I'$.
\smallskip

$\bullet$ In fact, $I=\overline I'$ \cite[Proposition~1.4.6]{MR2266432};
$I'$ is a minimal reduction of $I$, and its degree sequence
is $6\le 6\le 7\le 7\le 8$. The degree~$7$ appears twice in this sequence,
yet $-1/7$ is a simple pole of $\zeta_I(t)$.
So this is an example in which $Q_I(t)$ does
not equal the polynomial corresponding to the degree sequence of
a homogeneous minimal reduction.\smallskip

$\bullet$ The numerator $t+57t^2+640t^3+2016t^4$ has nonnegative
coefficients as prescribed in general by Theorem~\ref{thm:main}~(3).
In this case, the nonnegativity is further explained by the fact that
this polynomial is a combination of factors of the denominator, which
are products of terms $(1+d_i t)$, with coefficients given by 
{\em volumes\/} of certain simplices. 
\smallskip

It is tempting to guess that numerators of Segre zeta functions may
always be expressed in terms of volumes of certain polytopes in
Euclidean space, in analogy with what we have just verified in the 
monomial case.

\subsection{Ranks of a nonsingular projective variety}\label{ss:rks}
Let $\iota: Z\subsetneq\Pbb^n$ be a nonsingular projective variety. The
projective conormal variety $\Pbb(N^\vee_Z\Pbb^n)$ determines a
class of dimension $n-1$ in $\Pbb^n\times {\Pbb^n}^\vee$:
\[
[\Pbb(N^\vee_Z\Pbb^n)]=\delta_0(Z) H^n h + \delta_1(Z) H^{n-1} h^2 +
\cdots + \delta_m(Z) H^{n-m} h^{m+1} 
\]
where $H$, resp., $h$ are the pull-backs of the hyperplane classes
from $\Pbb^n$, resp., ${\Pbb^n}^\vee$, and $m=\dim Z$. (It is easy to see 
that the other terms in the decomposition of $[\Pbb(N^\vee_Z\Pbb^n)]$ vanish.)
The integers $\delta_i(Z)$ are the classical {\em ranks,\/} or {\em polar 
degrees\/} of $Z$, and have compelling geometric interpretations.
For example, $\delta_m(Z)$ is 
the degree of $Z$, while (in characteristic~$0$) the first nonvanishing 
$\delta_i(Z)$ is the degree of the dual variety $Z^\vee$ of $Z$, and the 
dimension of the dual is $n-1-i$ for the corresponding index~$i$. 
(See e.g., \cite[p.~152]{MR555696}, \cite[Theorem~1.1]{MR1861428}.)
It follows that $Z^\vee$ is a hypersurface if and only if 
$\delta_0(Z)\ne 0$.

\begin{prop}\label{prop:dual}
Let $Z\subsetneq\Pbb^n$ be a nonsingular subvariety of dimension~$m$, 
and let $I$ be a homogeneous ideal defining $Z$ in $\Pbb^n$. Then, with 
notation as above,
\begin{equation}\label{eq:ranks}
\zeta_I\left(-\frac t{1+t}\right) = (-1)^{n-m}\left(\delta_m(Z) t^{n-m} + \cdots 
+ \delta_0(Z) t^n + \text{higher order terms}\right)\,.
\end{equation}
\end{prop}

\begin{proof}
The definition of ranks implies that
\[
\delta_i(Z) = \int H^i \cdot s(N^\vee_Z\Pbb^n(H))\quad,
\]
see e.g., \cite[(1.4)]{MR1861428} (but note the different convention for Segre classes
used in this reference). Therefore,
\[
\iota_* (s(N^\vee_Z\Pbb^n(H))\cap [Z]) = \left( \delta_m(Z) H^{n-m} + \cdots 
+ \delta_0(Z) H^n \right) \cap [\Pbb^n]
\]
On the other hand, since $Z$ is nonsingular,
\[
\iota_* (s(N_Z\Pbb^n)\cap [Z]) = \iota_* s(Z,\Pbb^n)
=\zeta_I(H)\cap [\Pbb^n]
\]
in $A_*\Pbb^n$. Taking a dual amounts to changing the
sign of terms of every other codimension in the corresponding Segre class, hence
\[
\iota_* (s(N^\vee_Z\Pbb^n)\cap [Z]) = (-1)^{n-m} \zeta_I(-H)\cap [\Pbb^n]\quad.
\]
By~\cite[Proposition~1]{MR96d:14004},
\[
\iota_* (s(N^\vee_Z\Pbb^n\otimes \cO(H))\cap [Z]) 
= \iota_* (s(N^\vee_Z\Pbb^n)\cap [Z]) \otimes_{\Pbb^n} \cO(H)\quad,
\]
and hence
\[
\iota_* (s(N^\vee_Z\Pbb^n)\cap [Z]) \otimes_{\Pbb^n} \cO(H)
=(-1)^{n-m} \zeta_I\left(-\frac H{1+H}\right)\cap [\Pbb^n]
\]
since the effect of the operation $-\otimes_{\Pbb^n} \cO(H)$ on a class in 
$\Pbb^n$ is to replace $H$ by $H/(1+H)$. The statement follows. 
\end{proof}

\begin{example}
The twisted cubic $C$ in $\Pbb^3$ has degree~$3$ and admits an ideal $I$ 
generated by three quadrics. By Theorem~\ref{thm:main} (3) 
(and Remark~\ref{rem:less}), this is the only information needed to compute 
its Segre zeta function:
\[
\zeta_I(t) = \frac{3t^2+8t^3}{(1+2t)^3}\quad.
\]
By Proposition~\ref{prop:dual}, since
\[
(-1)^2 \zeta_I\left(-\frac{t}{1+t}\right) = \frac{3t^2-5t^3}{(1-t)^3} = 3t^2+4t^3+3t^4-5t^6-\cdots
\] 
we see $\delta_1=3,\delta_0=4$, and we conclude that the dual of $C$ is a
quartic surface.
\qede\end{example}

\begin{example}\label{exa:CI}
If $Z$ is a complete intersection of hypersurfaces of degrees $d_1,\dots, d_r$, then
\[
\zeta_I(t) = \frac{d_1\cdots d_r \, t^r}{(1+d_1 t)\cdots (1+d_r t)}
\]
where $I$ is the homogeneous ideal of $Z$ (Example~\ref{ex:compint}). 
By Proposition~\ref{prop:dual}, if $Z$ is nonsingular, then the ranks of $Z$ 
are given by the first several coefficients in
\[
(-1)^r \zeta_I\left(-\frac t{1+t}\right) = \prod_{i=1}^r \frac{d_i \frac{t_i}{1+t_i}}
{1-\frac{d_i t}{1+t_i}}=\prod_{i=1}^r \frac{d_i t_i}{1-(d_i-1) t}
=\prod_{i=1}^r d_i \sum_{j\ge 0} (d_i-1)^j t^{j+1}\quad.
\]
If $Z$ is not a linear subspace, i.e., some $d_i$ is greater than $1$, then 
{\em all\/} coefficients of $t^i$, $i\ge r$ in this series are positive. 
In particular, $\delta_0>0$. this verifies the well-known fact that the dual 
of a nonsingular complete intersection is necessarily a hypersurface. 
\qede\end{example}

The terms of higher order in $\zeta_I(t)$ or in the series at~\eqref{eq:ranks}
depend on the specific ideal chosen to cut out the subvariety, and reflect the
scheme structure deposited on the vertices of the corresponding cones in 
higher dimension. Since these cones are singular, the interpretation of these
higher-order coefficients as `ranks' is no longer valid. However, ranks may 
be defined for singular varieties, and are determined by the {\em Chern-Mather
class\/} of the variety (see~\cite{10,MR1074588,mather}). It is therefore 
natural to expect that these higher order terms record some `Chern-Mather'
information, and it would be interesting to obtain precise results of this type.

\subsection{Local complete intersections of small codimension in projective space}\label{sec:HC}
The Segre zeta function yields a transparent way to verify that, under suitable 
hypotheses, local complete intersections in projective space are actually
{\em global\/} complete intersections. Such results are motivated by 
Hartshorne's influential conjecture: {\em If $Z$ is a nonsingular 
subvariety of codimension $r$ of $\Pbb^n$, and if $r<\frac 13n$, then $Z$ 
is a complete intersection,\/} \cite{MR0384816}. To date, Hartshorne's conjecture
is still open even for subvarieties of codimension~$2$ in~$\Pbb^n$.

\begin{prop}\label{prop:Cc}
Let $\iota:Z\hookrightarrow \Pbb^n$ be a local complete intersection of codimension~$r$,
and assume that the Chern class of the normal bundle of $Z$ is the pull-back of a class 
from the ambient projective space $\Pbb^n$. Assume $Z$ is defined scheme theoretically
by a homogeneous ideal $I$ generated by $\le \dim Z$ forms. Then the Segre zeta 
function of $Z$ is the Segre zeta function of a complete intersection.
\end{prop}

We will in fact prove that if $Z=X_1\cap \cdots\cap X_m$, 
with $X_i$ hypersurfaces and $m\le \dim Z$, then (under the other hypotheses
in the statement of the proposition) $c(N_Z\Pbb^n)=\iota^* \prod_{i=1}^r 
c(N_{X_i}\Pbb^n)$ after a reordering of the hypersurfaces.

\begin{proof}
Let $I=(F_1,\dots, F_m)$, with $F_i$ homogeneous, and let $d_i=\deg F_i$.
We let $N_Z\Pbb^n$ be the normal bundle of $Z$ in $\Pbb^n$; this is a bundle 
of rank $r=\codim_Z\Pbb^n$. By hypothesis, 
\[
c(N_Z\Pbb^n) = \iota^*(1+c_1 H+ \cdots + c_r H^r)
\]
for some integers $c_1,\dots, c_r$, where $H$ is the hyperplane class in 
$\Pbb^n$. The Segre class of $Z$ in~$\Pbb^n$ is given by
\[
s(Z,\Pbb^n) = c(N_Z\Pbb^n)^{-1}\cap [Z]\quad,
\]
hence by the projection formula it pushes forward in $\Pbb^n$ to
\[
\iota_* s(Z,\Pbb^n) = \iota_*(\iota^*(1+c_1 H+ \cdots + c_r H^r)^{-1}\cap [Z])
=(1+c_1 H+\cdots + c_r H^r)^{-1}\cap (\deg Z) H^r\cap [\Pbb^n]\quad.
\]
On the other hand, 
\[
\iota_* s(Z,\Pbb^n) = \zeta_I(H)\cap [\Pbb^n]\quad,
\]
where $\zeta_I(H)$ is the Segre zeta function determined by $I$.
By the main theorem,
\[
\zeta_I(H) = \frac{P(H)}{(1+d_1 H)\cdots (1+d_m H)}
\]
where $P(H) = (\deg Z) H^r + \cdots + d_1\cdots d_m H^m$. (The
coefficient of $H^r$ equals $\deg Z$ because $Z$ is a local complete
intersection.) Therefore, we have
the equality of rational equivalence classes in $\Pbb^n$:
\[
(1+c_1 H+\cdots + c_r H^e)^{-1}\cap (\deg Z) H^r\cap [\Pbb^n]
=\frac{P(H)}{(1+d_1 H)\cdots (1+d_m H)}\cap [\Pbb^n]\quad,
\]
and hence
\begin{multline*}
(\deg Z)(1+d_1 H)\cdots (1+d_m H) H^r \cap [\Pbb^n] \\
= (1+c_1 H+\cdots + c_r H^r)
((\deg Z) H^r + \cdots + d_1\cdots d_m H^m)\cap [\Pbb^n]
\end{multline*}
in $A_*\Pbb^n\cong \Zbb[H]/(H^{n+1})$. Both sides of this identity are 
polynomials of degree $=r+m\le r+\dim Z = n$ by hypothesis.
It follows that the polynomials must be equal {\em in $\Zbb[H]$:\/}
\[
(\deg Z)(1+d_1 H)\cdots (1+d_m H) H^r = (1+c_1 H+\cdots + c_r H^r)
((\deg Z) H^r + \cdots + d_1\cdots d_m H^m)\quad.
\]
By unique factorization, we can conclude that
\[
(1+c_1 H+\cdots + c_r H^r) = (1+d_1 H)\cdots (1+d_r H)
\]
and
\[
(\deg Z) H^r + \cdots + d_1\cdots d_m H^m = (\deg Z)(1+d_{r+1} H)\cdots (1+d_m H) H^e
\]
after reordering the factors $(1+d_i H)$. Therefore, $\deg Z = d_1\cdots d_r$
and the Chern class of the normal bundle $N_Z\Pbb^n$ agrees with the
Chern class of a complete intersection of hypersurfaces of degrees $d_1,\dots, d_r$.
The result follows.
\end{proof}

By what we verified in Example~\ref{exa:CI} we can draw the following consequence.

\begin{corol}\label{cor:dd}
Let $Z\subseteq \Pbb^n$ be a nonsingular subvariety satisfying
the hypotheses of Proposition~\ref{prop:Cc}. Then the dual of $Z$ is a hypersurface.
\end{corol}

The hypothesis on 
the number of hypersurfaces cutting out $Z$ forces the codimension of~$Z$ to 
be `small', since if $Z$ is cut out by $\le \dim Z$ hypersurfaces, then necessarily 
$\dim Z\ge n-\dim Z$, i.e., $\codim_Z\Pbb^n\le \frac 12n$. Thus Corollary~\ref{cor:dd}
lends some support to the `Duality defect conjecture', stating that the dual of a
nonsingular subvariety of low codimension in projective space should be a hypersurface.

The hypothesis that the Chern class of the normal bundle of $Z$ is a restriction
from the ambient space holds automatically if $Z$ is nonsingular of codimension~$2$,
by \cite[Theorem~2.2, Proposition~6.1]{MR0384816}. Analogous results for smooth
varieties in higher codimension, subject to delicate inequalities, may be found 
in~\cite{MR970088}.

Under more stringent hypotheses, one can conclude that $Z$ is actually a global
complete intersection. For example, consider the following condition on a
scheme~$Z$ of codimension~$r$, cut out by $m$ hypersurfaces:
\begin{quote}
(*) There exist hypersurfaces $X_i$, $i=1,\dots, m$, such that 
$Z=X_1\cap\cdots\cap X_m$ and that $Z$ is a component of the intersection
$X_{i_1}\cap\cdots\cap X_{i_r}$ for all $r$-tuples of distinct hypersurfaces
$X_{i_1},\dots, X_{i_r}$.
\end{quote}

\begin{prop}\label{prop:Ci}
Let $Z\subseteq \Pbb^n$ be an irreducible and reduced local complete intersection 
satisfying the hypotheses of Proposition~\ref{prop:Cc} {\em and\/} condition (*).
Then $Z$ is a complete intersection.
\end{prop}

This follows from Proposition~\ref{prop:Cc} and the refined B\'ezout theorem
(\cite[Theorem~12.3]{85k:14004}).
Condition~(*) is verified if $\codim_Z{\Pbb^n}=2$, or if the $m$ hypersurfaces all
have the same degree, as can be verified easily.

Requirements such as condition (*) may be bypassed, again subject to certain
inequalities involving $r, m, n$ (see~\cite{MR604835}, \cite{MR970088}).


\end{document}